\documentclass[12pt]{article}
\usepackage{amsmath,amsfonts,amssymb,amsthm,amscd}
\title{On function field Mordell-Lang and Manin-Mumford}
\date{\today}
\author{Franck Benoist\thanks{Partially supported by ANR MODIG (ANR-09-BLAN-0047) Model
    theory and Interactions with Geometry and ANR ValCoMo
    (ANR-13-BS01-0006) Valuations, Combinatorics and Model Theory}
   \and Elisabeth
  Bouscaren\footnotemark[1] \and Anand Pillay\thanks{Partially supported by grants from the EPSRC and NSF} }
\newtheorem{Theorem}{Theorem}[section]
\newtheorem{Proposition}[Theorem]{Proposition}
\newtheorem{Definition}[Theorem]{Definition} 
\newtheorem{Definitions}[Theorem]{Definitions} 
\newtheorem{Remark}[Theorem]{Remark}
\newtheorem{Lemma}[Theorem]{Lemma}
\newtheorem{Corollary}[Theorem]{Corollary}
\newtheorem{Fact}[Theorem]{Fact}

\newtheorem{Claim}[Theorem]{Claim}

\newcommand{\F}{\mathbb F}
\newcommand{\C}{\mathbb C}

\newcommand{\FP}{\F_{p}^{\rm{alg}}} 
\newcommand{\conc}{^\curlywedge}
\newcommand{\CC}{\mathcal C}
\begin{document}

\maketitle

\begin{abstract} 
We give a reduction of the function field Mordell-Lang conjecture to the
function field Manin-Mumford conjecture, for abelian varieties,  in all characteristics, via
model theory, but avoiding recourse to the dichotomy theorems for
(generalized) Zariski geometries. Additional ingredients include the
``Theorem of the kernel", and a result of Wagner on commutative groups
of finite Morley rank without proper infinite definable subgroups. 
In positive characteristic, where the main interest lies, there is one more crucial ingredient: 
``quantifier-elimination" for the corresponding $A^{\sharp} =
p^{\infty}A(\cal U)$ where $\cal U$ is a saturated separably closed
field.
\end{abstract}

\section{Introduction}

This paper concerns relationships between ``known" results.  The original motivation was to supply a transparent account of the function field Mordell-Lang conjecture in positive characteristic.  The possibility of reducing Mordell-Lang to Manin-Mumford, in the function field case, was initiated in a talk  by the third author in Paris, in December 2010.  
Damian R\"ossler picked up the theme, and eventually with Corpet, produced a successful algebraic-geometric account of  such a reduction, in positive characteristic (\cite{R}, \cite{C}). Our  original strategy (quite different from what R\"ossler did) presented in the 2010 talk referred to above, involved some soft stable-group theory results, together with, in positive characteristic,  a conjectural ``quantifier-elimination" for certain type-definable groups.  In the current paper
 we prove that this strategy  works,  supplying the so far missing
ingredient in characteristic $p$, the ``quantifier elimination''
result for abelian varieties.   

The subtext is Hrushovski's proof of function field Mordell-Lang
\cite{Hrushovski}, which depends on a dichotomy theorem for
(generalized) Zariski geometries. In the characteristic $0$ case, it
is classical (strongly minimal) Zariski geometries which are relevant
and the dichotomy theorem is proved in \cite{HZ} and in \cite{Z}, although all proofs are complicated, to say the least.  But in the  positive characteristic case, {\em type-definable}  Zariski geometries are the relevant objects. The needed dichotomy in this case  is a combination of  a complicated axiomatic account of a field construction in \cite{HZ}, together with arguments in \cite{Hrushovski} showing that the axioms are satisfied for the particular minimal types in separably closed fields that we are interested in.  In both cases, the dichotomy theorem is difficult and its proof impenetrable for non model theory
  experts (as well as for many model-theorists). 
The current authors have been preoccupied for some years about seeing what is really going on, in particular avoiding the recourse to (generalized) Zariski geometries, and/or recovering the results by more direct arguments. 

In \cite{PZ} this issue was taken up, and an approach using differential jet spaces was developed. This succeeded in characteristic $0$, but not entirely in positive characteristic due to inseparability issues, although the approach recovered some cases due to Abramovich and Voloch \cite{AV}. It is still open whether the approach can be tweaked so as to work in general in the characteristic $p$ case.

In \cite{PR} Pink and R\"ossler gave a reasonably transparent algebraic-geometric proof of function field Manin-Mumford in positive characteristic with {\em all} torsion points in place of prime-to-$p$ torsion points.  This suggested to us to try to reduce function field Mordell-Lang to function field Manin-Mumford, but still using model-theoretic methods.  In the current paper we succeed in doing this, producing, so to speak, a ``second generation" model-theoretic proof of the Mordell-Lang  conjecture for abelian varieties over function fields. 
 Exploring connections between the methods of R\"ossler and Corpet and either the jet space ideas in \cite{PZ}, or  the current paper, would be interesting.  We should  also mention the interesting paper  \cite{Paul-Ziegler} in which another new algebraic-geometric account of Mordell-Lang in positive characteristic is given, but where only finitely generated groups $\Gamma$ are considered. 

Of course we could also consider the absolute Mordell-Lang conjecture in characteristic $0$ (proved by Faltings, McQuillan,..) and ask whether there is a ``soft" reduction to absolute Manin-Mumford. We doubt that this is the case as these two theorems seem to us (maybe incorrectly) to be of different orders of difficulty. That such a reduction is possible in the function field case has additional interest.

\vspace{5mm}
\noindent
{\em Acknowledgements.} We would like to thank  both Universit\'e
Paris-Sud, Orsay, where the third author was a Professeur Invit\'e in
March-April 2010 and the Mathematical Sciences Research Institute,
Berkeley, where the authors participated in the Spring 2014 model theory
program. The first author would like to thank Fran\c{c}oise Delon for
very useful discussions about the quantifier elimination question. Thanks also to the referee of
a first version of this paper for helpful comments, resulting in a reorganization of the paper.

\section{Preliminaries}\label{preliminaries}

The statements of the function field Mordell-Lang which we prove in this paper, as well as the main arguments we use, are specific to abelian varieties. They are also restricted to the function field case in 
one variable ($\C(t)^{{\rm alg}}$ in characteristic $0$, and
$\F_p(t)^{{\rm sep}}$ in characteristic $p$)  as  the 
``Theorem of the Kernel'' has been proved only in
these cases, in \cite{BP} and  in \cite{Roessler-torsion}. Basic
background, as well as references,  on abelian varieties, Mordell-Lang
and Manin Mumford can be for example found in \cite{hindrymanchester}.

\vspace{5mm}
\noindent
{\bf Statement of function field Mordell-Lang  in characteristic $0$.}
Let $K = \C(t)^{\rm{alg}}$, the algebraic closure of $\C(t)$. Let $A$ be an abelian variety over $K$ with $\C$-trace $0$. Let $X$ be an irreducible subvariety of $A$ (defined over $K$), and let $\Gamma$ be a ``finite-rank'' subgroup of $A(K)$, namely $\Gamma$ is contained in the division points of a finitely generated subgroup of $A(K)$. Suppose $X\cap \Gamma$ is Zariski-dense in $X$. Then $X$ is a translate of an abelian subvariety of $A$.

\vspace{5mm}
\noindent
{\bf Statement of function field Manin-Mumford in characteristic $0$.} As above, except that the hypothesis on $\Gamma$ is strengthened to: $\Gamma$ is contained in the group of torsion points of $A$.

\vspace{5mm}
\noindent
{\bf Statement of function field Mordell-Lang in characteristic $p>0$.} Let $K$ be the separable closure of $\F_{p}^{\rm{alg}}(t)$. Let $A$ be an abelian variety over $K$ with $\FP$-trace $0$. Let $X$ be an irreducible subvariety of $A$, defined over $K$. And let $\Gamma$ be a subgroup of $A(K)$ contained in the prime-to-$p$-division points of a finitely generated subgroup. Suppose that $X\cap \Gamma$ is Zariski-dense in $X$. Then $X$ is a translate of an abelian subvariety of $A$. 

\vspace{5mm}
\noindent
The formulation involving prime-to-$p$ division points is due to
Abramovich and Voloch \cite{AV}. One could also ask what happens when
$K$ is algebraically closed and  $\Gamma < A(K)$ is the group of {\em all}  division points of some finitely generated subgroup. No obstacle is currently known to this.

\vspace{5mm}
\noindent
{\bf Statement of function field  Manin-Mumford  in characteristic $p$.} As above, except that $\Gamma$ is assumed to be contained in the group of {\em all torsion points} of $A$.

\vspace{5mm}
\noindent
We will write $MM$ for Manin-Mumford and $ML$ for Mordell-Lang.
And from now on, $k$ will denote $\C$ in characteristic zero and $\FP$ in characteristic $p$, $K = k(t)^{\rm{alg}}$ in characteristic zero, and $K = k(t)^{sep}$ in characteristic $p$.

In characteristic $0$, function field $MM$ as stated  is clearly a 
special case of function field $ML$. And it follows from the absolute
case of $MM$, of which there are many proofs more generally for all commutative algebraic groups (for example see \cite{hindryMM} or, for a model theoretic proof, \cite{Hrushovski-MM}). In positive
characteristic $MM$ (with all torsion points) is proved by Pink and
R\"ossler \cite{PR}. The proof uses a variety of methods, including 
Dieudonn\'e modules, but is accessible. A proof was also given by
Scanlon \cite{Sc} using  the dichotomy theorem in $ACFA_{p}$ (algebraically closed fields of characteristic $p$ with
a generic automorphism) which itself depends on an even more generalized notion of Zariski geometries than used in separably closed fields.

\medskip 

\noindent 
As explained in the introduction, {\bf {\em in all
    characteristics},  assuming function field Manin-Mumford,  we give a
  proof of function field Mordell-Lang  {\em  without appealing to the
    dichotomy theorems for Zariski geometries}.}

The beginning of our proof  follows the  first steps of Hrushovski's original proof
in \cite{Hrushovski}, and has the
uniform strategy in both characteristic zero and characteristic $p$ of ``embedding" the algebraic-geometric set-up in a differential algebraic environment.  But then 
we replace the use of the dichotomy  theorems by some other ingredients which  appeal to ``softer'' model theory. \\
Two of these ingredients are  common to all characteristics, as we will explain below: the first one, of an algebraic nature, we refer to as the ``Theorem of the Kernel'' ; the second one, pertaining to model theory, is  the structure of $g$-minimal groups of finite Morley rank. 

Then, in characteristic zero, like Hrushovski himself in his proof, we use the 
``weak socle theorem"  for groups of finite Morley rank (see Section \ref{Char.0}). 

In characteristic $p$, we do not need  this socle theorem for groups, but we
  need a new result, which forms the core of the paper:   in
 section \ref{QEsection}, we prove quantifier elimination for the
 induced structure on the subgroup of infinitely $p$-divisible
 $K$-rational points of the abelian variety $A$, and this of course is accomplished
 without appealing to the dichotomy theorems.

 We assume  familiarity with model theory, basic stability, as well as
differentially and separably closed fields.  The book \cite{MMP} is a
reasonable reference, as well as \cite{Pillay-book} for more on
stability theory. Definability means with parameters unless we say
otherwise.  

\subsection{Passing to the differential framework, the group $A^\sharp$}\label{sharp}

As in Hrushovski's approach, (as well as in Buium's in  \cite{Buium-Annals} for the characteristic $0$ case), we pass 
to a differential algebraic framework: differentially closed fields of
characteristic zero, or 
 separably closed fields of degree of imperfection one, in
 characteristic $p$, which can be viewed as the existentially closed
 fields with Hasse-Schmidt derivations. 
In these enriched frameworks, one can ``replace'' the group $\Gamma$ by some (infinitely) definable subgroup of the rational points of $A$. \\
We fix some more notation. 

\noindent 
In characteristic $0$, $K$ has a unique derivation $\partial$
extending $d/dt$ on $k(t)$ and $K^{\rm{diff}}$ denotes a differential
closure of $(K,\partial)$. We work in the language of differential
fields,  with the first order theory of $K^{\rm{diff}}$, the theory of
differentially  closed fields of characteristic zero, $DCF_{0}$. It will
be  convenient sometimes to work in a saturated elementary extension
$\cal U$ of $K^{\rm{diff}}$. Recall that $DCF_{0}$ is $\omega$-stable
with quantifier-elimination and elimination of imaginaries.

\noindent In characteristic $p$, $\cal U$ will be a saturated elementary
extension of $K$ in the language of fields. In contradistinction to the characteristic $0$ case, passing to
 $\cal U$ will be a crucial part of the proof, rather than just a
convenience. The first order theory of $K$ (or $\cal U$)  in the language of fields is
known as or denoted by $SCF_{p,1}$, the theory of separably closed
fields of characteristic $p$ and degree of imperfection $1$. It is
stable, but not superstable, and has quantifier-elimination and
elimination of imaginaries after either adding symbols for a $p$-basis
and the so-called $\lambda$-functions, or just symbols for a strict iterative Hasse-Schmidt derivation. 

From now on, $A$ is an abelian variety over $K$.

\begin{Definition}{\em{The group $A^\sharp$}.}
\newline (i) In characteristic $0$, $A^{\sharp}$ denotes  the ``Kolchin
closure of the torsion", namely the smallest definable (in the sense of
differentially closed fields) subgroup of $A(\cal U)$ which contains the
torsion subgroup (so note that $A^{\sharp}$ is definable  over $K$.)
\newline
(ii) In positive characteristic, $A^{\sharp}$ denotes
$p^{\infty}(A({\cal U})) =_{def} \bigcap_{n}p^{n}(A(\cal U))$ 
(an infinitely definable subgroup over $K$. 
\end{Definition}

\begin{Remark} {\rm In  characteristic $0$, the smallest definable
    subgroup containing the torsion subgroup exists by
    $\omega$-stability of the theory  $DCF_{0}$.  As $A^\sharp $ is
    definable, we can consider not only $A^\sharp(\cal U)$ , but also
    $A^\sharp (K^{\rm{diff}})$.  In  positive characteristic,
    $A^{\sharp}$ is only infinitely definable, it is the maximal
    divisible subgroup of $A({\cal U})$ and it is also the  smallest
    infinitely definable subgroup of $A(\cal U)$ which contains the
    prime-to-$p$ torsion of $A$.  Moreover $A^{\sharp}(K) = \bigcap_{n}p^{n}(A(K))$ 
    is infinite and  is  also the maximal divisible subgroup of $A(K)$.}    \end{Remark}

An exposition of further properties of $A^\sharp$ in both
characteristics, as well as precise references  can be found in
\cite{BBP}.

 \noindent For example, the following basic facts hold in all
 characteristics and were mostly originally  proved in \cite{Hrushovski}: 

\begin{Fact} \label{FactAsharpbasic}(i) $A^\sharp $ is also the   smallest Zariski dense (infinitely)  definable subgroup of $A(\cal U)$.\\
(ii) $A^{\sharp}$ is connected (no relatively definable subgroup of finite index), and of finite $U$-rank in char. $p$, and finite Morley rank in char. $0$.\\
(iii) If $A$ is a simple abelian variety, $A^\sharp$ has no proper
  infinite 
type-definable subgroup.\\
(iv)  If $A$ is the sum of simple $A_{i}$ then $A^{\sharp}$ is the sum of the $A_{i}^{\sharp}$.\\
\end{Fact}

We can now state:

\vspace{5mm}
\noindent
{\bf The  Theorem of the Kernel.} {\em In all characteristics: Let $A$ be an abelian variety
  over $K$ with $k$-trace $0$, then  $A^{\sharp}(K)$ is contained in the group of torsion points of $A$.}

\vspace{2mm}
\noindent
This is a differential algebraic or model-theoretic theorem of the kernel. In Corollary K3 of \cite{BP}, the characteristic $0$ case is given, where it is deduced from Chai's strengthening of a theorem of Manin.
In \cite{Roessler-torsion}, the positive characteristic case is
proved. Note that in the positive characteristic, the hypothesis of
$k$-trace zero is not needed. 

\vspace{5mm}

\noindent
We now pass to the purely  model-theoretic ingredients.

\subsection{Groups of finite Morley rank and induced structures}

\begin{Definition} Let $G$ be a group (with additional structure), which has
  finite Morley rank and is commutative and connected.  We say that $G$
  is $g$-{\em minimal} if it has no proper nontrivial  connected definable
  subgroup (equivalently, no proper infinite definable subgroup).
\end{Definition} 

Let us remark that in this context, $g$-minimality of $G$ passes to saturated elementary
extensions: In groups of finite Morley rank, there is a bound on the
cardinality of uniformly definable families of finite subgroups, as they
do not have the ``finite cover property''(see for example \cite{Poizatbook}). 

\smallskip 
\noindent 
The following appears in \cite{W} (as a direct consequence of Corollary 6):

\begin{Theorem} \label{wagner} Suppose that $G$ is $g$-minimal. Then any infinite algebraically closed subset of $G$ is (the universe of) an elementary substructure of $G$. 
\end{Theorem}

\noindent
{\bf Remark}: So a $g$-minimal group behaves  like a strongly minimal set, where it is well known that infinite algebraically closed subsets are elementary substructures. By Zilber's indecomposability theorem a $g$-minimal group is 
{\em almost strongly minimal}, namely in the algebraic closure of a
definable strongly minimal  subset $D$  together with a finite set $F$  of
parameters. Then it is also well-known that any algebraically
closed (over $F$)
subset whose intersection with $D$ is infinite  is an elementary
substructure. But here we only suppose that $X$ is  algebraically closed
(over $\emptyset$)  and  infinite, so one cannot directly deduce the result from the almost strongly minimal case.

\vspace{5mm}
\noindent
We now fix an ambient saturated stable structure $\cal U$. A subset $X$
of ${\cal U}^{n}$ is called {\em type-definable} (or infinitely
definable)  if it is the
intersection of a small   collection of definable sets, namely defined by a small partial type. In the case of interest $X$ will be a countable intersection of definable sets.

Suppose $X$ is type-definable over the small set of parameters $A$. By a
{\em relatively definable} subset $Y$ of $X^{m}$ we mean the
intersection of some definable (with parameters) subset $Z$ of a
suitable Cartesian power of ${\cal U}$ with $X^{m}$. We will say
that $Y$ is relatively $A$-definable, or relatively definable over $A$,
if $Z$ can be chosen to be $A$-definable.

\smallskip  

\noindent Let us fix a set $X$, type-definable  over some small set of parameters $A$.

\begin{Lemma}\label{invariant} Let $Y$ be a relatively definable subset
  of $X$ (or some Cartesian power of $X$) which is invariant under
  automorphisms of 
$\cal U$ fixing $A$ pointwise. Then $Y$ is relatively $A$-definable.
\end{Lemma}
\begin{proof} Suppose $Y$ is relatively definable by formula $\phi(x,b)$  (namely $Y$ is the set of solutions in $X$ of $\phi(x,b)$), where we are exhibiting the required parameters $b$.  By our assumptions, we have that
if $tp(b_{1}/A) = tp(b/A)$ then $\phi(x,b_{1})$ relatively defines the same subset of $X$ as does $\phi(x,b)$. We can apply compactness to find a formula $\psi(y)\in tp(b/A)$ such that $Y$ is relatively defined by the formula $\exists y(\psi(y)\wedge\phi(x,y))$.

\end{proof}

Let us denote by ${\mathcal X}_{A}$  the structure with universe $X$
and predicates for all relatively definable over $A$  subsets of $X^{n}$, for all $n$.  With this notation:

\begin{Definition} We will say that $X_{A}$ has {\em quantifier elimination} or QE,  if $Th({\cal X}_{A})$ has quantifier elimination in the language above. 
\end{Definition}

It is clear from this definition  that:
\begin{Remark}  $X_{A}$ has QE if and only if whenever $Y$ is a relatively $A$-definable subset of $X^{n+1}$ then the projection of $Y$ to $X^{n}$ is relatively $A$-definable.
\end{Remark}

Let us recall a few basic facts about induced structures: 

\begin{Lemma} \label{inducedQEfacts} (i) Suppose $A\subseteq B$. Then $X_{A}$ has QE iff $X_{B}$ has QE  (so  we just say that $X$ has QE). 
\newline
(ii) $X$ has QE just if the projection of any relatively definable subset of any $X^{n+1}$ to $X^{n}$ is relatively definable (noting that in general it is only type-definable).
\newline
(iii) $X$ having QE is equivalent to ${\cal X}_{A}$ being a saturated structure.

\end{Lemma}
\begin{proof} (i) Right implies left follows from Lemma \ref{invariant} as a
  projection of an $A$-invariant set is also $A$-invariant.  For left to
  right: Suppose $Y\subset X^{n+1}$ is relatively definable 
over $B$, by $\phi(x_{1},..,x_{n+1},b)$ where we witness the parameters $b\in B$, which may live outside $X$.  Let $Y_{1}$ be the projection of $Y$ to $X^{n}$. By Lemma \ref{invariant} it 
suffices to prove that $Y_{1}$ is relatively definable (as clearly it is
invariant under automorphisms of $\cal U$ fixing $b$ pointwise). Now by
definability of types we may find some 
$L$-formula $\psi(x_{1},..,x_{n+1}, z)$  and $c$ from $X$ such that $Y$ is relatively definable by $\psi(x_{1},..,x_{n+1},c)$.  Let  $z$ be  an $m$-tuple of variables. Let $Y'$ be the 
subset of $X^{n+1+m}$ relatively defined by $\psi$. By our assumption that $X_{A}$ has QE, the projection $Y''$ of $Y$ obtained by existentially quantifying out $x_{n+1}$, is a relatively 
$A$-definable subset of $X^{n+m}$,  (relatively) defined by a formula $\chi(x_{1},..,x_{n},z)$ say. So $\chi(x_{1},..,x_{n},c)$ relatively defines $Y_{1}$, as required. 
\newline
(ii) By (i) and Lemma \ref{invariant}
\newline
(iii)  The point is that to say that $X_{A}$ has QE means precisely that if $b_{1}, b_{2}$ are $n$-tuples from $X$, then $tp(b_{1}/A) = tp(b_{2}/A)$ in the sense of ${\cal U}$ iff $tp(b_{1}) = tp(b_{2})$ in the sense of ${\cal X}_{A}$.

\end{proof}

Note that in general ${\cal X}_{A}$ is only quantifier-free saturated (and homogeneous) and is sometimes referred to as a Robinson structure. \\

Of course when $X$ is definable (rather than type-definable) over $A$, then $X_{A}$ always has QE and is  referred to as ``$X$ with its induced structure". \\

Using Theorem \ref{wagner} applied to the $G_i$'s with their induced
structure, we deduce easily:
\begin{Corollary}\label{Corwagner} Suppose $G$ is a (saturated) commutative, connected, group of finite Morley rank (with additional structure) which is a sum of finitely many $g$-minimal $\emptyset$-definable subgroups $G_{i}$. Then any  algebraically closed subset of $G$ which meets each $G_{i}$ in an infinite set, is an elementary substructure of $G$.
\end{Corollary}

\bigskip

In positive characteristic, in order to apply \ref{wagner} and Corollary
\ref{Corwagner}, to $A^\sharp$, we will need to show that $A^\sharp$ has
quantifier elimination, in the sense above. For example, we know
(\ref{FactAsharpbasic}) that if $A$ is a simple abelian variety,
  then $A^\sharp $ has no proper relatively definable subgroup,, but we will
  need to know  that the group $A^\sharp$, equipped with predicates for relatively definable sets,  is a $g$-minimal group, as a first order structure in its own right; QE says that every definable subgroup of this first order structure $A^\sharp$ is  relatively definable (in $\cal U$), which is exactly what we need.

Now the (generalized)
Zariski geometry arguments  from \cite{Hrushovski} give the dichotomy
theorem for minimal ``thin" types in separably closed fields, implying
that if $A$ is simple with $k$-trace $0$, then $A^{\sharp}$ is
minimal (i.e. $U$-rank $1$), connected, and $1$-based, from which it
easily follows that $A^{\sharp}$ has QE. For arbitrary traceless abelian
varieties $A$, $A^{\sharp}$ will be a sum of such minimals, hence also
$1$-based and  so we also
have QE. So the QE hypothesis is true, after the fact so to speak, once one knows $A^\sharp $ is one-based.

 Our method of proving QE is related to
proofs that minimal thin types in $SCF$ are Zariski, but does not use
the dichotomy theorem for (type-definable) Zariski geometries (nor in
fact  even the so called dimension theorem). In fact we prove QE for
$A^\sharp$ directly, for {\em any abelian variety} over a separably
closed field of imperfection degree $1$, without any assumption on the
trace. The fact that $A$ is an abelian variety, and not a semiabelian
variety plays an essential role, though. Indeed,  we can define $G^{\sharp}$ in the same way for
semiabelian varieties $G$. But building on work in \cite{BBP}, Alexandra
Omar Aziz, in \cite{OA} gave semiabelian examples $G$ for which
$G^{\sharp}$ does not have QE.

\section{Quantifier elimination for $A^\sharp$ in characteristic $p$}\label{QEsection}

Here we prove the quantifier elimination result in full generality. As
this is the case we need, and it simplifies notation, we keep the
assumption that we are working with a separably closed field of degree of
imperfection one, but the same proof will work  for any finite non
zero degree of imperfection.   

\begin{Theorem} \label{QEthm}
Let $A$  be an  abelian variety over any  separably
closed field $K$ of degree of imperfection $1$, and $\cal U$ be a saturated extension of 
$K$. We consider $A^\sharp=p^{\infty} A({\cal U})$, and we denote by $\cal A$ the structure $A^\sharp$ with relatively definable sets (with 
parameters from $K$). Then $\cal A$ has quantifier elimination.
\end{Theorem}

\noindent Note that by Lemma \ref{inducedQEfacts}, it will follow that ${\cal A}_{\cal U}$ has quantifier elimination as well. But, conversely, Lemma \ref{inducedQEfacts} says that, in order to prove the theorem for $\cal A$, {\em  we can suppose, for this proof only,  that $K$ itself is sufficiently 
saturated}. This will be used  in Claim \ref{notorthogonal}. \\

We will work in the language of rings augmented by $\lambda$-functions and a constant $t$ for the $p$-basis (see \cite{Delon} as reference).
It means that for each $a$ in the field, $a=\sum_{i\in p^n} \lambda_{n,i}(a)^{p^n}M_i$, where elements of $p^n$ are seen as sequences of length $n$ with values between $0$ and $p-1$, and for $i\in p^n$ and $j\in p$, $M_j=t^j$, $M_{i\conc j}=M_iM_j^{p^n}$ and $\lambda_{n+1,i\conc j}=\lambda_{1,j}\circ \lambda_{n,i}$. The tuple $a_{=n}:=(\lambda_{n,i}(a))$ is called the tuple of $\lambda$-components of level $n$, and we denote $a_{\le n}=(a_{=m})_{0\le m \le n}$, $a_{\infty}=(a_{=n})_{n\ge 0}$.\\
Similarly, for $X$ a multivariable, $X_{\infty}=(X_i)_{i\in p^n, n\ge 0}$. Let $I^0(X)$ be the ideal of $K[X_\infty]$ generated by the polynomials $X_i-\sum_{j\in p}X_{i\conc j}^pM_j$, for all $i\in p^n$, $n\ge 0$. A complete type over $K$ corresponds bijectively to a prime separable ideal of $K[X_\infty]$ containing $I^0(X)$. An ideal $I$ of $K[X_\infty]$ is separable if and only if whenever $\sum_{j\in p} P_j^p M_j$ is in $I$, each $P_j$ is in $I$.\\
We will use the $\Lambda_n$ functors, from the category of algebraic varieties (over $K$) into itself, as defined in \cite{BouscarenDelon1}. For each algebraic variety $V$, we have a natural definable map $\lambda_n:V({\cal U}) \to \Lambda_nV({\cal U})$ and a natural morphism $\rho_n :\Lambda_n V \to V$ which define reciprocal bijections between $V({\cal U})$ and $\Lambda_nV({\cal U})$ (the notation $\lambda_n$ is coherent with the previous one for the affine line). Note that $\Lambda_n$, $\lambda_n$ and $\rho_n$ are obtained from $\Lambda_1$, $\lambda_1$ and $\rho_1$ by the suitable composition. We equip each $V({\cal U})$ with the $\lambda$-topology, whose basic closed sets are $\lambda_m^{-1}(V_m)$ for any $m\ge 0$ and $V_m$ any algebraic subvariety of $\Lambda_m V$. Note that for any $n\ge 0$, $\Lambda_n V_m$ is an algebraic subvariety of $\Lambda_{m+n}V$, and $\lambda_m^{-1}(V_m)=\lambda_{m+n}^{-1}(\Lambda_n V_m)$. It follows that finite unions of basic closed sets are still basic closed; we obtain arbitrary closed sets by possibly infinite intersection. It follows from quantifier elimination in the theory of separably closed fields in this language that definable subsets of $V({\cal U})$ are boolean combinations of definable (i.e. basic) closed sets.\\  
If $A$ is an abelian variety, $\Lambda_m A$ is an algebraic group, but is not an abelian variety. But 
from \cite{BenoistDelon}, we know that for each $m\ge 0$, there is an algebraic subgroup $A_m \subset \Lambda_m A$, which is isogenous to $A$, such that $p^mA({\cal U})=\lambda_m^{-1}(A_m)$ (it was shown there using the formalism of Weil's restriction of the scalars $\Pi_{K/K^{p^m}}A$, which is known to coincide with $F^{(m)}\Lambda_m A$, $F^{(m)}$ being the $m$-th power of the absolute Frobenius).

\medskip 
Recall that $A$ is a sum of simple abelian varieties, $A=\Sigma_{1\leq i\leq n}  A_i$. 
We will  use the following consequence of the Zilber
indecomposability theorem (see Fact 3.8 in \cite{BBP}): as $A_i$ is
simple, ${A_i}^\sharp $ has no proper infinite type-definable subgroups (Fact \ref{FactAsharpbasic}) and hence, there exists a minimal type in ${A_i}^\sharp$
whose set of realizations, which  will be denoted by $Q_i$ in the
following, is such that ${A_i}^\sharp= Q_i+\ldots+Q_i$ ($m_i$
times for some $m_i$, that is, any element of ${A_i}^\sharp$ is the sum
of $m_i$ elements from $Q_i$).

\medskip 

First we recall some basic facts about the group 
$A^\sharp$ in characteristic $p$, which will be used here and in Section
\ref{Char.p}

We refer the reader to \cite{BBP} where we give a precise account of
relative Morley rank for type-definable sets in a stable structure
(called ``internal Morley dimension" in \cite{Hrushovski}).

\begin{Fact}\label{basicAsharpcharp} Both $A^{\sharp}$ and the
  $A_{i}^{\sharp}$ are connected groups,
  with finite relative Morley rank. Moreover, $A_{i}^{\sharp}$ is the
  connected component of $A^\sharp \cap A_{i}$ and is relatively
  definable in $A^{\sharp}$. 
\end{Fact} 
\noindent
{\em Comments.}    The fact that $A^\sharp $ has finite relative Morley
rank is  claimed in Remark 2.19 of \cite{Hrushovski}. However it is also
implicitly claimed there that  $G^{\sharp} = p^{\infty}(G(\cal U))$ also
has finite relative Morley rank, whenever $G$ is {\em semiabelian}, and
this is actually wrong, as pointed out in \cite{BBP}. So we refer the
reader rather to the proof of Fact 3.8 from \cite{BBP}. It is worth
remarking that it is this  semiabelian counterexample without relative
Morley rank which is shown, in 
\cite{OA},  not to have QE. But in general there is no reason why a
type-definable group of finite relative Morley rank should have QE. 
Once we know that $A^\sharp$ has relative   finite Morley rank, as, for each $i$,
$A_{i}^{\sharp}$ is the connected component of $A_{i}\cap A^{\sharp}$,
it follows that $A_{i}^{\sharp}$ is relatively definable in
$A^{\sharp}$, but again this is no longer true in the semiabelian 
  counterexample.\\

\begin{Proposition} \label{Noetherian} 
For every $n\geq 1$, the $\lambda$-topology on  $(A^\sharp)^n$ is Noetherian of finite dimension.
\end{Proposition}

Note that, in the examples of  semiabelian varieties which do not have
finite relative finite Morley rank given in \cite{BBP},  the topology is not Noetherian.

It was  shown in \cite{Hrushovski}  that the trace of the
$\lambda$-topology on each $Q^k$ is Noetherian and has finite dimension, if $Q$ is a {\em thin} minimal type (see below), which is the case here. Now the Noetherianity for ${A_i}^\sharp$
follows as we have a continuous relatively definable surjective map from 
${Q_i}^{md}$ onto $ {{A_i}^{\sharp}}^d$.

\medskip 

\noindent Passing to $A^\sharp$ itself will be a little more complicated.

\medskip 
\noindent Let us first note that, once we know that the topology is
Noetherian,  it is easy to see that it is finite dimensional, that is,
that every closed set has finite topological
dimension. Indeed, by Noetherianity, every closed set is a finite union
of irreducible closed sets, and in $(A^\sharp)^n $,  there is a finite bound on the length of
strictly decreasing  sequences of {\em irreducible} closed sets. This follows
from ``thinness'' of the types involved. Recall that if $a\in (A^\sharp)^n $ then the type of $a
$ over $K$ is {\em thin}, that is, the field generated over $K$ by $a$
and its images under the $\lambda$-functions has finite transcendence degree
over $K$ (\cite{Hrushovski} or \cite{BD}). More precisely, the
transcendence rank of the prime separable ideal $I(a/K)$, of all
$\lambda$-polynomials vanishing at $a$, has finite transcendence rank
smaller than $(dim A) ^n$.  
Now if $F\subsetneq G\subseteq (A^\sharp)^n $ are two irreducible closed sets
and $I(G) \subsetneq I(F)$ are the associated  two prime separable  ideals  
then, the transcendence rank of both are finite and  bounded by $(dim A) ^n$
and the transcendence rank of $I(F)$ must be strictly smaller than the
transcendence rank of $I(G)$. 

\medskip

So, in order to prove  Proposition \ref{Noetherian},   we just need to show Noetherianity.

\begin{Lemma}\label{cartesian} Let $p_1,\ldots, p_k$ be thin minimal types over $K$ and
  for each $i$, let $P_i$ denote the set of realizations of $p_i$ in
  $\cal U$. Then for every $n_1,\ldots ,  n_k$ the topology on 
${P_1}^{n_1} \times  \ldots \times {P_k}^{n_k}$ is Noetherian of
  finite dimension. In particular every closed
  subset is relatively definable.
\end{Lemma}

We proceed in two steps.  

\begin{Claim}\label{notorthogonal} Suppose $q_1,\ldots q_n$ are pairwise non  orthogonal minimal
  thin types (not necessarily distinct) and $Q_j$ denotes the set of realizations of $q_j$ in
  $\cal U$. Then there is a minimal thin type over $K$, $r$, such that, if
  $R$ denotes the set of realizations of $r$ in  $\cal U$, then there
  is a continuous map $f$ from $R^{n}$ onto ${Q_1}\times 
  \ldots \times {Q_n}$. It follows
  that the topology on ${Q_1}\times 
  \ldots \times {Q_n}$ is Noetherian. 
\end{Claim}

\begin{proof} 
Once we have shown that such a type $r$ and the required continuous maps
exist, the Noetherianity  follows directly
from that of Cartesian products of $R$, which was shown in \cite{Hrushovski}. 
 
\noindent Let $(a_1,a_2,\ldots ,a_n)$, realizing $q_1\times q_2\times \ldots \times
q_n$, be such that $r := tp(a_1,a_2,\ldots, a_n)$ has U-rank equal to
one. This exists by the assumption of non pairwise orthogonality and the
saturation assumption on $K$.
So the type $r$ is minimal and is easily seen to be also thin (the
transcendence degree of the field  generated by the $a_i$ and their
images by the $\lambda$-functions will be finite as each $q_i$ has this
property). 

For each $i$, the ${i}$-th component map, $\pi_i$  from $R$ to $Q_i$ is surjective, as
all elements of $Q_i $ realize the same type over $K$. 

Now consider the Cartesian  product $R ^n$. We claim that there is a
surjective definable map $f$ from $R^n$ onto $Q_1 \times \ldots \times Q_n$ and
that it is continuous.
Let $(b_1,\ldots b_n) \in R^n$, and let $f(b_1,\ldots ,b_n) := (\pi_1
(b_1),\ldots , \pi_n(b_n)) \in Q_1\times\ldots \times Q_n$. 
 Let $(a_1, \ldots , a_n)$ be any tuple from $Q_1
\times \ldots \times  Q_n$. For each $i$ there is some element $b_i$ from $R$ such
that $\pi_i (b_i) = a_i$, so $f$ is surjective.

Now $f$  is the restriction to $R^n$  of a projection map $(Q_1\times \ldots \times Q_n)^n \to Q_1 \times \ldots \times Q_n$, so it is continuous.
\end{proof}

Note that as we have not supposed that the non-orthogonal types in
Claim \ref{notorthogonal} were distinct, this gives the Noetherianity for
any Cartesian  product of finitely many non pairwise orthogonal minimal
thin types.

\begin{Claim} \label{orthogonal} Suppose $q_1,\ldots q_m$ are pairwise orthogonal minimal
  thin types, and for each $i$, $1\leq i \leq m$, $\{p_{(i,1)},\ldots
  ,p_{(i,n_i)}\}$ is a set 
of minimal (thin) types non orthogonal to $q_i$.
  Let $P_k$ denote the set of realizations of $p_k$ in $\cal U$, and let   
  ${Q_i}$ denote the cartesian product ${P_{(i,1)}}
      \times \ldots \times {P_{(i, n_i)}}$.   Then the topology on   ${Q_1}\times 
  \ldots \times {Q_m}$ is Noetherian, and every closed set $C$ is
  a finite union of closed sets of the form $W_1 \times \ldots \times
  W_m$, where $W_i$ is a closed subset of ${Q_i}$. 
\end{Claim}
\begin{proof}

First we note that by pairwise orthogonality and minimality of the
types $q_j$ and $p_k$, if $\bar a_j$ is 
a tuple of elements from
${Q_j}$, then $\{\bar
a_1,\ldots, \bar a_m\}$, is an independent set of tuples over $K$: indeed by
minimality, each $\bar a_j$ is contained in the algebraic closure
(over $K$) of a
$K$-independent subtuple $B_j$,   each element in $B_j$ realizing one of
the types  $\{p_{(j,1)},\ldots
  ,p_{(j, n_j)}\}$, which are all non-orthogonal to $q_j$.  By orthogonality
of the $q_j$, 
the set   $B_1 \cup \ldots \cup B_m$ form an independent set over $K$ and the
rest follows.  

A first observation: Let   $Z= Z_b \subseteq  {Q_2}\times  \ldots
\times {Q_m}$ be  a closed
set (hence infinitely definable) defined over $K \cup b $ where
$b$ is a tuple of elements in $Q_1$.  Then $Z$ is in fact
defined over $K$. Indeed $Z$ is $K$-invariant: let $b'$ realize the
same  type as $b$ over $K$. By the remark above $b$ and $b'$ must be
independent from  $Q_2 \cup \ldots \cup Q_m$ over $K$, it follows that $b$ and
$b'$ also realize the same type over $K \cup Q_2 \cup \ldots \cup
Q_m$, hence $Z_b = Z_{b'}$. 

Now we show by induction on $m\geq 1$ that   the topology on
${Q_1}  \times \ldots \times {Q_m}$ is Noetherian, and that every closed set is of the required form. 

The case $m =1$, that is, Cartesian  products of a finite set of pairwise non-orthogonal
 thin minimal types, is Claim \ref{notorthogonal}, which says that the
 topology is Noetherian.  

Now, for $m>1$,  let $C\subseteq {Q_1}\times   \ldots \times {Q_m}$ be any
closed subset, and let $a:= (a_1, a_2, \ldots , a_m) \in C$. We show
that 

(*) $a$ is contained in a set  of the form $W \times Z \subseteq C$ 
where $W$  is a closed subset of ${Q_1}$ defined over $K$ and $Z$
is a closed subset of ${Q_2} \times \ldots \times {Q_m}$, also
defined over $K$. 

The Noetherianity and form of the closed sets will follow: by the induction assumption, $W$ is a
relatively  definable closed subset in ${Q_1}$ and $Z$ is a relatively
definable closed subset in ${Q_2} \times \ldots \times {Q_m}$ of
the right form. By (*),
$C$ which is infinitely definable (over  $K$) is  covered by a union
of such relatively definable sets each of the form $W \times Z$ each defined
over $K$. By (model theoretic) compactness, it follows that $C$ is a
finite union of such sets. And as ${Q_1}$ and ${Q_2}\times 
\ldots \times {Q_m}$ are both Noetherian, the
result follows. 

It remains only to check condition (*). 

Let $Z := Z_{a_1} = \{ (a_2,\ldots ,a_m) \in 
{Q_2}\times \ldots \times {Q_m} : (a_1, a_2, \ldots ,
a_m) \in C\} $. The set $Z$ is closed and defined over $K \cup
a_1$. By the first observation made above,  $Z$ is in fact defined over $K$. 
Now let $W = \{ x \in {Q_1} : \mbox{ for all } y \in Z, (x,y)\in
C\}$.
Then $W$ is closed (it is the intersection of the $C_y$, for $y\in Z$)
and clearly defined over $K$. 

And of course $a \in W \times Z$ and $W \times Z \subseteq C$.

\end{proof}

We can now conclude the proof of  Lemma \ref{cartesian}: 
\noindent For each non orthogonality class represented amongst the $p_i$, choose a
representative $q$ and apply Claim \ref{orthogonal}.\qed

\medskip 

\noindent
{\em We now go back to the proof of Proposition \ref{Noetherian}}: 
Let $A$ be any abelian variety over $K$, $A=\Sigma_{1\leq i\leq n}
A_i$, where each $A_i$ is a simple abelian variety over $K$. It follows
that $A^\sharp = \Sigma_{0\leq i\leq n} {A_i}^\sharp$. For each $i$
there is a minimal thin type $q_i$ such that $ {A_i}^\sharp = Q_i+
\ldots + Q_i$   ($m_i$ times). So $A^\sharp = Q_1+\ldots +Q_1 +Q_2+\ldots
Q_2 + \ldots + Q_n+ \ldots + Q_n$. 

So we have a continuous relatively definable surjective map, for every
$d$  from the
Cartesian product $({Q_1}^{m_1} \times \ldots \times {Q_n}^{m_n}) ^d $ onto
${A^\sharp}^d$.  And Noetherianity now
follows from Lemma \ref{cartesian}, and this concludes the proof of
Proposition \ref{Noetherian}. \qed

\medskip

\noindent Note, as a corollary, that each closed set of $(A^\sharp)^n$ is actually relatively definable.

\medskip

If $\mathcal{U}_0 \prec \cal U$, and $a\in \cal U$,
$\mathcal{U}_0 [a_\infty]$ denotes  the ring generated by $a$ and its images by the
$\lambda$-functions over $\cal U_0$
and 
$\mathcal{U}_0(a_\infty)$, the fraction field of $\mathcal{U}_0 [a_\infty]$.

\begin{Lemma}\label{basictopology} Let $X$ and $Y$ be closed sets in
  $(A^\sharp)^d$, with $X$ irreducible, and $pr:X\times Y \to X$ the
  projection. Let $G\subsetneq F$ be closed subsets of $X\times Y$, such that $\overline{pr(F)}=X$ and $F$ is irreducible. Let $a$ be a topological generic of $X$ over some small model $\mathcal{U}_0$ of definition for $F$ and $G$, and we denote $F(a)=\{y\in Y \mid (a,y) \in F\}$, and similarly for $G$.
Then $G(a)\subsetneq F(a)$. Moreover $F(a)$ is irreducible as a closed set over $\mathcal{U}_0(a_\infty)$ (note however that it may be reducible as a closed set over $\cal U$).
\end{Lemma}
\begin{proof}
We denote by $\mathcal{U}_0[X]_\infty:=\mathcal{U}_0[T_\infty]/I(X)$ the ring of $\lambda$-coordinates of $X$ over $\mathcal{U}_0$, and $\mathcal{U}_0[Y]_\infty$ in a similar way. By irreducibility, $\mathcal{U}_0[X]_\infty$ is an integral domain, and by choice of $a$, $\mathcal{U}_0[X]_\infty \simeq \mathcal{U}_0[a_\infty]$.\\
We denote by $I(F)$ and $I(G)$ the separable ideals in $\mathcal{U}_0[a_\infty][Y]_\infty \simeq \mathcal{U}_0[X\times Y]_\infty$ corresponding to $F$ and $G$. Since $\overline{pr(F)}=X$, $I(F)\cap \mathcal{U}_0[a_\infty]=0$. Now $I(F(a))$ is the ideal generated by $I(F)$ in $\mathcal{U}_0(a_\infty)[Y]_\infty$, and similarly for $I(G(a))$. We claim that $I(F)\subsetneq
I(G)$ implies that $I(F(a)) \subsetneq I(G(a))$. 
If $I(F(a))=I(G(a))$, then for every $P\in I(G)$, there is some non zero
$d\in {\cal U}_0[a_\infty]$ such that $dP\in I(F)$, which implies that
$P\in I(F)$ since $d\not \in I(F)$, which is prime. That contradicts
$I(F)\neq I(G)$. We get that $I(F(a))$ is prime by the same kind of
argument. This  means that $F(a)$ is irreducible as a closed set over $\mathcal{U}_0(a_\infty)$. 
\end{proof}

We now prove that $A^\sharp$ is complete in the category of $\lambda$-closed varieties, in the following sense:

\begin{Proposition} \label{closed}
Let $F$ be a definable $\lambda$-closed subset of $(X\times A)({\cal U})$, where $X$ is an algebraic variety over $K$, and $pr:X\times A \to X$ the projection. Then $pr((X({\cal U})\times A^\sharp) \cap F)$ is $\lambda$-closed.
\end{Proposition}

We first need the following lemmas.

\begin{Lemma} Let $V,W$ be two   affine varieties defined
  over $K$. Let $g$ be a morphism  from $V$ to $W$ defined
  over $K$.\\
 
Let $b \in W({\cal U})$. Then $b \in g(V({\cal U}))$ if and only if, for every $n$, 
 $ \lambda_n (b) \in (\Lambda_n g) ( \Lambda_n V)$. 
 \end{Lemma}

\noindent \begin{proof} 

If $V \subset \mathbb{A}^r$ and $W \subset \mathbb{A}^k$,
then $g  =(g_1,\ldots g_k)$, where each $g_j$ is a polynomial map 
from $V$ to $\mathbb{A}^1$, and $\Lambda_m g_j$ is a $p^m$-tuple of elements in $K[X_{=m}]$ for $X$
a multivariable of length $r$.

One direction is clear: if $b=g(a)\in g(V({\cal U}))$, this follows from the
fact that $\lambda_n \circ g (a)=\Lambda_n g \circ \lambda_n (a)$.\\

\smallskip 

Now for the other direction, let  $b\in W({\cal U}) $ be such that  for every
  $n$, $ \lambda_n (b) \in  (\Lambda_n g) (\Lambda_n  V)$. 

Let $H$ be the separable closure of $K(b_\infty)$, $K \preceq
H \preceq {\cal U}$, and consider the following set $R$ of polynomials: 
$$  R \,  := \, \{(\Lambda_m g(X_{=m}))_\sigma  - b_\sigma \, ; \sigma  \in
  p^m, m \geq 0 \}.$$ 
Let $J_0 $ denote the ideal  of the variety $V$ over $H$, and  $I$ in $H[X_\infty]$, be the ideal generated by
$I^0$, $J_0  $ and $R$.  
     
\smallskip

\noindent{\bf Claim 1} $I\neq H[X_\infty]$

\smallskip

\noindent {\em Proof of Claim 1}: It suffices to find an infinite sequence $e_\infty$ in $\bar{\cal U}$ such that $I(e_\infty) \supset I$. By saturation of ${\cal U}^{\rm alg}$ as an algebraically closed field, it suffices to find, for each $n$, a sequence $e_{\le n}$ such that
$I(e_{\le n}) \supset I_{=n}:=I \cap H[X_{=n}]$. From the hypothesis on $b$, there is $a_n \in (\Lambda_n V)(\bar {\cal U})$ such that $(\Lambda_n g)  (a_n) = \lambda_n(b)$. Now for $0\le m \le n$, set $e_{=m}=\rho_{n-m}(a_{=n})$, and  it gives the required tuple $e_{\le n}$.

\medskip 

\noindent {\bf Claim 2} The ideal $I$ is separable.  

\smallskip 

\noindent {\em Proof of Claim 2}: Using the fact that $I^0 \subset I$, it can be easily shown that it suffices to show that for each $h\in I_{=n}:=I \cap H[X_{=n}]$, and each $j<p$, $(\Lambda_1 h)_j \in I$ (where $(\Lambda_1 h)_j$ is the $j^{th}$ component of $\Lambda_1 h$ viewed as a tuple of elements in $H[X_{=n+1}]$). Since such  $h$ can be written as $h=h_0+h_1+\sum_k P_k Q_k$ with
$h_0 \in I^0$, $h_1 \in J_0$ (which are separable), and $P_k \in H[X_{=n}]$, $Q_k \in R \cap  H[X_{=n}]$, and since we have the relations
$(\Lambda_1(f_1+f_2))_j = (\Lambda_1 f_1)_j + (\Lambda_1 f_2)_j$ and $(\Lambda_1 f_1f2)_j = \sum_{k+l-m=j}(\Lambda_1 f_1)_k (\Lambda_1 f_2)_l M_m$, it suffices to consider the case of $Q \in R \cap  H[X_{=n}]$. But
this is  clear by the definition of $R$: indeed, 
if $Q = (\Lambda_n g(X_{=n}))_\sigma  - b_\sigma$ for some $\sigma \in p^n$, $(\Lambda_1 Q)_j = (\Lambda_{n+1} g (X_{=n+1}))_{\sigma \conc j} - b_{\sigma \conc j}$.

\medskip 
Now since $I$ is a proper separable ideal containing $I^0$, there is $a\in {\cal U}$ such that $I(a_{\infty}) \supset I$, which implies that
$a\in V({\cal U})$ and $g(a)=b$.

\end{proof}

\medskip

The following result, a direct corollary of the previous Lemma,  is certainly well-known, probably in the formalism of Weil restrictions.
\begin{Lemma} \label{image}
Let $V,W$ be two irreducible varieties defined
  over $K$. Let $g$ be a morphism  from $V$ to $W$ defined
  over $K$. Then there is some $N$ such that 
 $$x \in g(V({\cal U})) \, \mbox{ iff }\, \lambda_N (x) \in ((\Lambda_N g) (\Lambda_N V))
 ({\cal U})  ).$$
\end{Lemma}
\begin{proof}  Going from affine varieties to arbitrary algebraic varieties comes from the fact that the functors $\Lambda_n$ preserve open affine coverings. Now we know that $g(V({\cal U}))$ is a definable set in the
 separably closed field $\cal U$. By the previous  Lemma it is also
 infinitely definable as the infinite conjunction of the formulas 
$ \lambda_k (x) \in (\Lambda_k g) (\Lambda_k V)$. The result then
 follows by compactness.\end{proof}

\begin{proof}[Proof of Proposition \ref{closed}.]
For some $m$ and some subvariety $F_m$ of $\Lambda_m (X\times A)$, $F=\lambda_m^{-1}(F_m)$. We also have for any $n\ge m$ that 
$F=\lambda_{n}^{-1}(F_n)$, where $F_n=\Lambda_{n-m}F_m \subset \Lambda_n (X\times A)$.
Let $x$ be in $X({\cal U})$. By saturation, $x\in pr((X \times A^\sharp) \cap F)$ iff for every $n\ge m$, there is some $z_n \in p^nA({\cal U})$ such that $(x,z_n)\in F$, which in turn is equivalent to $(*)_n$: there is some $y_n\in A_n({\cal U})$ such that $(\lambda_n(x),y_n)\in F_n$. We now use Corollary \ref{image}: there is some $l=l_n$ (which depends on $n$ but not on $x$) such that $(*)_n$ is equivalent to $\lambda_l \circ \lambda_n(x) \in (\Lambda_l pr)(\Lambda_l F_n)$, where $\Lambda_l pr:\Lambda_{l+n} (X\times A) \simeq \Lambda_{l+n} X \times \Lambda_{l+n} A \to \Lambda_{l+n} X$ is still the projection on the first factor. Now look at the condition $(**)_n$: $\lambda_l \circ \lambda_n(x) \in pr(\Lambda_l F_n \cap (\Lambda_{l+n} X \times A_{l+n}))$. Recall that $A_{l+n}\subset \Lambda_{l+n}A$ characterizes $p^{l+n}$-divisible points. Since $A_{l+n}$ is complete, $(**)_n$ is a $\lambda$-closed condition. If $x$ satifies $(**)_n$, it satisfies $(*)_n$ a fortiori. Conversely, if $x\in pr((X \times A^\sharp) \cap F)$, the $z_n$'s appearing in the previous discussion can be chosen to be $p^\infty$-divisible, which in particular implies that we can find $y_n$ such that $\lambda_l(y_n)=\lambda_l \circ \lambda_n (z_n) \in A_{l+n}$, hence that $(**)_n$ is satisfied.\\
We conclude that $pr((X\times A^\sharp) \cap F)$ is given by the conjunction of the conditions $(**)_n$, hence is $\lambda$-closed.  
\end{proof}

From Proposition \ref{closed} and Proposition \ref{Noetherian}, we get:

\begin{Corollary} \label{corclosed}
Let $F$ be a definable closed subset of $A^d$ and $pr:A^d \to A^{d-1}$ the projection on the first coordinates. Then $pr(F({\cal U})\cap (A^\sharp)^d)$ is closed and relatively definable.
\end{Corollary}

In order to go from the case of closed sets to the case of constructible
sets, we really follow the lines of the proof of quantifier elimination
for one-dimensional  Zariski
geometries given in \cite{Z} or \cite{HZ} (note that QE for one
dimensional Zariski geometries is a basic
consequence of the axioms and does not involve the deep dichotomy
result).  We know  that $A^\sharp =  Q_1+\ldots +Q_1 +Q_2+\ldots +
Q_2 + \ldots + Q_n+ \ldots + Q_n$, where $Q_i $ is the set of
  realizations of a thin minimal type.  It
  would be convenient to work with the relatively definable closed sets
  $\overline Q_i$, 
  the closure of $Q_i$ in 
  the sense of the $\lambda$-topology,  but it is not clear  a priori
  why it should be of topological dimension $1$. A more general fact
  would be that the $U$-rank of a type $t$ coincides with the topological
  dimension of $\overline t$ (the closure of the set of its
  realizations, or equivalently, the closed set given by the type ideal
  corresponding to $t$). It is true for types in $A^\sharp$, actually,
  but we know it only a posteriori, via
  the dichotomy theorem, which we do not want to use.   We know however
  that $U$-rank $(t)\leq dim(\overline t)$: because of the correspondence between closed type-definable sets and prime separable ideals in the suitable polynomials algebra, $dim(\overline t)$ is given by the separable depth of the corresponding ideal $I_t$, and the separable depth is a stability rank, hence greater than or equal to the $U$-rank (see \cite{Delon88}). 

It follows that if $F \subseteq A^\sharp$  is  irreducible closed of
topological dimension one, then $F$  has a unique type of U-rank
one, which is the type of maximal rank and is also its topological generic. 

So, rather than $Q_i$, we will consider  suitable relatively definable
irreducible closed sets $H_i$ of dimension $1$. We proceed as follows. 
 For each $i$, let $H_i$ be a relatively definable irreducible closed
 subset of ${A_i}^{\sharp}$ of dimension $1$ (it exists since we know
 that the topology is Noetherian, and since translations are
 homeomorphisms, we are allowed to replace $H_i$ by any of its
 translates in the following). By the comparison between the $U$-rank
 and the topological dimension above, it is clear that the generic
 type (in the topological sense) of $H_i$ is a minimal type, hence, as 
${A_i}^{\sharp}$ has no proper infinite type-definable subgroups, we can apply Zilber's indecomposability theorem to get $A_i^\sharp = H_i+\ldots+H_i$ ($m_i$ times for some $m_i$).\\  
Hence $A^{\sharp} = \sum_{m_1}H_{1} + \ldots + \sum_{m_n}H_{n}$,
where  the $H_{i}$ are relatively definable closed subsets of $A^{\sharp}$ of
(topological) dimension $1$. \\
Now for any formula $\phi(x,\overline a)$ with parameters in $A^\sharp$, $\mathcal A \models \exists x \, \phi(x,\overline a)$ if and only if 
$$\mathcal A \models \exists x_{1,1}\in H_{1}\ldots \exists x_{1,m_1}\in
H_{1} \ldots \exists x_{n,1}\in H_{n} \ldots \exists x_{n,m_n}\in H_{n}
\, \phi(\sum_{i,j}x_{i,j},\overline a).$$ 
Hence it is sufficient to
consider projections of the form $pr:H\times (A^\sharp)^d \to
(A^\sharp)^d$, where $H$ is one of the $H_{i}$'s.   
From Corollary
\ref{corclosed} and the fact that $H$ is closed, we get that $pr$ takes
closed sets to closed sets. From quantifier elimination in the separably
closed field $\cal U$ and Noetherianity in $\mathcal A$, we just have to consider projections of definable sets $F\setminus G$, where $G\subsetneq F \subseteq H\times (A^\sharp)^d$ are closed relatively definable sets, with $F$ irreducible.

\begin{Proposition} The projection $pr(F\setminus G)$ is constructible in $\mathcal A$.
\end{Proposition}
\begin{proof}
We proceed by induction on $dim(F)$. The case $dim(F)=0$ is obvious since it implies that $F$ is a singleton.\\
Now $dim(F)=k+1$. We consider the closed sets $F_1=pr(F)$, $G_1=pr(G)$, $F_0=\{\overline y\in (A^\sharp)^d \mid \forall x \in H, (x,\overline y) \in F\}=\bigcap_{x\in H} F_x$, where $F_x=\{\overline y\in (A^\sharp)^d \mid (x,\overline y) \in F\}$ is closed (note that we allow parameters in $A^\sharp$ in the definition of the topology), and $G_0=\{\overline y\in (A^\sharp)^d \mid \forall x \in H, (x,\overline y) \in G\}$. There are three cases:
\begin{enumerate}
\item if $F_0=F_1$, we see easily that $pr(F\setminus G)=F_0\setminus G_0$, hence is constructible.
\item if $G_1 \subsetneq F_1$, we have a proper closed subset $(H\times G_1)\cap F \subsetneq F$, hence $dim((H\times G_1)\cap F)<dim(F)$ since $F$ is irreducible. But we can write $pr(F\setminus G)=F_1 \setminus G_1 \cup pr(((H\times G_1)\cap F) \setminus G)$, and the result comes from the induction hypothesis.
\item if $F_0 \subsetneq F_1 =G_1$, we consider a generic point $a$ of
  $F_1$ over $\mathcal{U}_0$, a model of definition for $F$ and $G$
  (note that $F_1$ is irreducible since $F$ is). In particular $a\not
  \in F_0$, which implies that the fibre $F(a)$ is finite, as it is a
  proper closed set of  irreducible dimension one $H$. Furthermore, by Lemma \ref{basictopology}, $G(a)\subsetneq F(a)$, and $F(a)$ is irreducible as a closed set over $\mathcal{U}_0(a_\infty)$. It follows that $F(a)$ is the orbit of any of its points under 
$\text{Aut}({\cal U}/\mathcal {U}_{0}(a_{\infty}))$ (note that ${\mathcal U}_0(a_\infty)$ is definably closed in $\cal U$, see \cite{Delon}), and a fortiori, it is the orbit of $G(a)$. Hence $F(a)=G(a)$, a contradiction.
\end{enumerate}
\end{proof}

So we have proved Theorem \ref{QEthm}.

\section{Deriving  function field ML from function field MM}

\subsection{The characteristic $0$ case}\label{Char.0}

We first deal with the characteristic $0$ case. \\

We work with the notation introduced in Section
\ref{preliminaries}: $k = \mathbb C$, $K = k(t)^{\rm{alg}}$,
$K^{\rm{diff}}$ is a  differential closure of the differential field
$(K,d/dt)$ and $\cal U$ a saturated elementary extension of
$K^{\rm{diff}}$. In fact because the  theory $DCF_0$ is $\omega$-stable,
we will mostly be working in $K^{\rm{diff}}$, without needing to go up
to the big model $\cal U$, contrary to the characteristic $p$ case where
it is essential, in the  course of the proof,  to go up to a sufficiently saturated model. 

\noindent As explained before, we want to deduce Mordell-Lang from Manin-Mumford
and the Kernel Theorem, avoiding reference to the dichotomy theorem for strongly minimal sets in differentially  closed fields and/or the local modularity and  strong minimality of $A^{\sharp}$ when $A$ is simple with $k$-trace $0$. But we will use relatively softer ingredients of Hrushovski's proof in \cite{Hrushovski}, among them the  {\em weak socle theorem}, which we recall below.  

\begin{Definitions} {\rm 1.} Let $G$ be a commutative connected group of finite
  Morley rank, definable in some ambient stable structure $M$. 
We say that $G$ is generated (abstractly) by some sets $X_1,\ldots, X_n$
if $G \subset acl ( F \cup  X_1\cup \ldots \cup X_n)$, where $F$ is a
finite set. If $G \subset acl (F \cup X)$, where $X$ is a strongly
minimal set, $G$ is said to be {\em almost strongly minimal}.\\
{\rm 2.} The model-theoretic (or stability-theoretic) {\em socle} $S(G)$ of $G$,
is the greatest connected definable subgroup of $G$ which is generated  by strongly minimal definable subsets of $G$.\\ 
{\rm 3.} For $X$ a definable subset of $G$ with Morley degree $1$, define the model theoretic {\em stabilizer} of $X$ in $G$, $Stab_{G}(X)$,  to be $\{g\in G:MR( X\cap (X+g)) = MR(X)\}$ (so the stabilizer of the generic type of $X$ over $M$).\\
{\rm 4.} A definable subgroup of $G$, $H$, defined over some $B$,  is said to be {\em rigid} if, passing to a saturated model, all connected definable subgroups of $H$ are defined over $acl(B)$. 
\end{Definitions}

We will need to know some properties of socles:

\begin{Fact} \label{socleproperty} {\rm (Lemma 4.6 in \cite{Hrushovski},
    or Section 7 in \cite{lascar})} The group $S(G)$ 
is an almost direct sum of pairwise 
  orthogonal definable groups $G_i$, where each $G_i$ is almost strongly
  minimal.  
\end{Fact}

In this context, Hrushovski's so called ``weak socle theorem'' can be
stated as:

\begin{Proposition} \label{socletheorem} ({\rm (Prop. 4.3 in \cite{Hrushovski} or Prop; 2.10 in \cite{bouscaren})} Let G be a commutative connected
  group of finite Morley rank. Suppose that the socle of $G$, $ S(G)$,
  is rigid. Let $X$ be a definable subset of $G$ of Morley degree one, such that  $Stab_{G}(X)$ is finite. Then, up to subtracting a definable subset of $X$ of strictly smaller Morley rank than that of $X$,  some translate of $X$ is contained in $S(G)$.
\end{Proposition}  

The reason we call the above Proposition the weak socle theorem, is because a stronger statement is proved in \cite{Pillay-ML} (see Theorem 2.1 there) in the special case of algebraic $D$-groups, and we have sometimes refered to the latter as ``the socle theorem". \\

Let $A$ be an  abelian variety over $K$ and $A^\sharp := A^\sharp(\cal U )$, the smallest definable subgroup of $A(\cal U )$ which contains the torsion of $A$ (see Section \ref{sharp}).

We have already seen some basic properties of $A^\sharp$ (Fact \ref{FactAsharpbasic}). Recall that $A^\sharp$ is definable over $K$,  connected with finite Morley rank  and that if $A$ is simple, $A^\sharp$ is a g-minimal group. 

We are going to use the following now  classical fact (see for example 6.5 in
\cite{ziegler}): 

\begin{Fact}\label{nonorthogonal} Let $H$ be an almost strongly minimal  definable group in $\cal U$, non orthogonal to the field of constants $\CC$. Then there is an algebraic group $R$ over $\CC$ and a definable isogeny from $H$ onto $R(\CC)$.\end{Fact}

The following, true in characteristic $0$,  are also well known, except
maybe (iv).

\begin{Lemma} \label{Asharpchar0} 
(i) $A({\cal U})/A^{\sharp}$ embeds definably (in the sense of $DCF_0$) in (the group of ${\cal U}$-points of) a unipotent algebraic group over $\cal U$. 
\newline
(ii) $A^{\sharp}$ is rigid. 
\newline 
(iii) If $A$ is simple, then $A$ has $k$-trace $0$ if and only if it is not isogenous to an abelian variety over $k$ if and only if $A^\sharp $ is orthogonal to the field of constants $\CC$.
\newline  
(iv) If $H$ is a connected finite Morley rank definable subgroup of $A = A({\cal U})$ containing $A^{\sharp}$, then $S(H) = A^{\sharp}$.
\end{Lemma}

\noindent {\em Remark}: We will only use (iv) with the extra assumption that $A$ has $k$-trace $0$,
but we give below the proof for the general case. 

\begin{proof} The following fundamental fact, due to Buium
  (\cite{Buium-book}), and  also proved in \cite{Pillay}, is essential:
\newline
(*) Suppose $G = G({\cal U})$ is a commutative connected algebraic group over $\cal U$ and $H$ is a Zariski-dense definable subgroup. Then $G/H$ definably embeds in a unipotent algebraic group, namely $({\cal U},+)^{d}$ for some $d$. 
\newline
(i)   is given by (*).
\newline
(ii) It follows for example from \cite{BP}, section 3.1, that any definable connected
subgroup $H$ of $A^\sharp$ is of the form $B^\sharp$, for some abelian
subvariety $B$ of $A$. Hence it is the Kolchin closure of the torsion
of $B$ which is contained in the  torsion of $A$, hence in $A(K)$. This
shows that $H$ is itself also defined over $K$. 
\newline
(iii) The main point is to show that if $A^{\sharp}$ is nonorthogonal to the field of constants, then  $A$ is isomorphic to an abelian variety over $k$. (The rest follows from the definition of trace $0$ and simplicity of $A$.)  This is already contained in \cite{Hrushovski} but we recall the proof.  Applying Fact 4.4 and the dual isogeny we obtain a definable isogeny $f$ from $R({\cal C})$ to $A^{\sharp}$. By QE in $DCF_{0}$, $f$ is given by a rational function, and hence extends to a a surjective homomorphism  of algebraic groups from $R$ to $A$, which we again call $f$. Now $R$ is a connected commutative algebraic group over $\cal C$, so as $\cal C$ is algebraically closed, and $Ker(f)$ contains the unique maximal connected  linear algebraic subgroup of $R$, $R/Ker(f)$ is again defined over $\cal C$, and isomorphic as an algebraic group to $A$. 
\newline
(iv) This reduces  to the case where $A$ is simple.  Now $A^{\sharp}$ is $g$-minimal,
hence, by Zilber's indecomposability theorem, generated by some strongly
minimal subset, so $A^\sharp = S(A^\sharp)$ and is contained in $S(H)$. 

Suppose by way of contradiction
that $S(H)$ properly contains $A^{\sharp}$. Then by (*)
$S(H)/A^{\sharp}$ is  an infinite  finite-dimensional vector space over the
constants $\mathcal C$ of $\cal U$.\\ 
If $A^\sharp $ has $k$-trace $0$, then  by (iii) $ A^\sharp$ is orthogonal to $S(H)/A^{\sharp}$, and so, by \ref{socleproperty},   $S(H) = A^\sharp + V$ where $V$ is infinite, orthogonal to $A^\sharp$ and has  finite intersection with $A^\sharp$. As $A$ is simple, $V$ is Zariski dense in $A$, which contradicts  
Fact \ref{FactAsharpbasic} (i) which says that $A^\sharp$ is the smallest Zariski dense definable subgroup of $A$.\\ 
  If $A$ has non zero $k$-trace, as  $k$ is algebraically closed and we are in characteristic zero, $A$ is isomorphic to an abelian variety over $k$, so we can suppose that $A$ itself is defined over $k$. It follows that $A^\sharp = A(\CC)$. Now as both $A^\sharp$ and $S(H)/A^\sharp$ are non orthogonal to the strongly minimal  
field of constants,  $S(H)$ is almost strongly minimal, $S(H) \subset acl (F \cup \CC)$.  It follows by \ref{nonorthogonal} and  arguments which are now standard  that  $S(H)$ is definably isomorphic to $R(\CC)$ for $R$ a connected algebraic group over $\CC$ (see for example  in \cite{Hrushovski}  or   \cite{bouscaren} the proof of Proposition 2.1). Now let $N$ be the biggest connected linear subgroup of $R$, $N$ is defined over $\CC$ also, and we know that $R/N$ is an abelian variety. Pulling back the situation by the definable isomorphism  it follows that there is some infinite  definable group $V$ such that $V\subset S(H)$ and  $V \cap A^\sharp$ is finite. As $A$ is simple, $V$ must be Zariski dense in $A$, but this contradicts Fact \ref{FactAsharpbasic} (i) again.  
\end{proof}

\noindent {\em Remark about (ii)}: In \cite{Hrushovski} it is shown that for any $H$,
definable connected subgroup of $A(\cal U) $ of finite Morley rank, $S(H)$ is
rigid. We do know that $A^\sharp = S(A^\sharp )$  because 
$A^\sharp = \Sigma_{0\leq i\leq n} {A_i}^\sharp$, where each  $
{A_i}^\sharp$ are $g$-minimal, hence generated by a strongly minimal
set. But we do not want to quote this result from \cite{Hrushovski}
because its proof uses the full dichotomy theorem for strongly minimal
sets in $DCF_0$. In fact, the easy proof we use here is really the same as the one given
in \cite{Hrushovski}  to show that any abelian variety is rigid. 

\vspace{2mm}
\noindent
{\em Remark on stabilizers.} In the proof below, 
when we speak of stabilizers of (differential) algebraic subvarieties of a (differential) algebraic group we mean set-theoretic stabilizers. But in the contexts we consider this coincides with the model-theoretic stabilizer defined above. Likewise, if $X$ is an irreducible differential algebraic subvariety of a finite Morley rank differential algebraic group $G$, then the conclusion of 4.3 says that $X$ is, up to translation, {\em contained in}  $S(G)$.

\vspace{2mm}
\noindent
{\bf Proving Mordell-Lang} 
We now suppose that $A$ has $k$-trace $0$. 
The first steps are exactly the same as the ones in
\cite{Hrushovski}, which were inspired by Buium in
\cite{Buium-Annals}. First, quotienting by the connected component of 
$Stab_{A}(X)$ we
obtain another abelian variety over $K$ with $k$-trace $0$. So, we will  assume that $Stab_{A}(X)$ is finite (note that in the end when we get to the conclusion that $X$ is the translate of a subabelian variety of $A$, the fact that the stabilizer is finite means that $X$ is just one point).  
 
 By (*) in the proof of Lemma 2.2,  $A/A^{\sharp}$ definably embeds via some $\mu$ in a vector group. So
 $\mu(\Gamma)$ is contained in a finite-dimensional vector space over $\cal C$, the preimage of which we call $H$: a connected 
definable finite Morley rank subgroup of $A$ containing both $A^{\sharp}$ and $\Gamma$, and defined over $K$. As $X\cap \Gamma$ is Zariski dense in $X$, so is  $ X\cap H$.  One reduces to the case when $X\cap H$ is irreducible as a differential algebraic variety. It follows that $Stab_{H}(X\cap H)$  is finite. So, by  Lemma \ref{Asharpchar0} (ii) and (iv), 
we can apply  the weak  socle theorem, Proposition \ref{socletheorem}, and,  after replacing $X\cap H$  (and
so $X$)  by a suitable translate (which will be defined over $K^{\rm{diff}}$), $X\cap H$ is contained in  $A^{\sharp}= S(H)$. 

From then on, we diverge from the proof in \cite{Hrushovski} which at this point appealed to the dichotomy. 

We now consider $\mathcal A $, the $K$-definable group
$A^{\sharp}(\cal U)$ with all its induced structure over $K$, namely equipped with predicates for $K$-definable subsets of Cartesian powers. As $A^\sharp$ is definable,  we automatically have quantifier elimination for  $\mathcal A$ and of course $\mathcal A (K^{\rm{diff}})$ is an elementary substructure of $\mathcal A$. Furthermore, as $A^\sharp$ is a definable group of finite Morley Rank, 
so is $\mathcal A $.   Now consider the subset $\mathcal A (K)$ of
$\mathcal A $. One can check easily that it must be algebraically closed
in $\mathcal A$. Now $A^\sharp $ is a finite sum of $K$-definable $g$-minimal subgroups $A_{i}^\sharp$. And we likewise have the ${\mathcal A}_i$, the $A_{i}^{\sharp}$ with the induced structure.  
Each $A_{i}^{\sharp}(K)$ is clearly  infinite as it contains the torsion of $A_i$, so by  Corollary \ref {Corwagner},
$\mathcal  A (K) $ is an elementary substructure of $\mathcal A$.  
 \newline
The following was proved in \cite{marker-pillay},  
with the  extra assumption that $A^\sharp $ was strongly minimal: 

\noindent {\em Claim.}  $A^\sharp  (K)  = A ^\sharp (K^{\rm{diff}})$. 
\newline
{\em Proof of claim.}  If $b\in A^{\sharp}(K^{\rm{diff}})$ then $tp(b/K)$ is isolated in $DCF_{0}$. 
But then $tp(b/ \mathcal A (K))$ is isolated in the structure $\mathcal A$. As $\mathcal A (K)$ is an elementary substructure of $\mathcal A $, $b\in A^{\sharp}(K)$. 

\vspace{2mm}
\noindent

Note that $X$ is now defined over $K^{\rm{diff}}$, and $X\cap A^{\sharp}(K^{\rm{diff}})$ is Zariski-dense in $X$.  By  the Claim,
 $A^{\sharp}(K^{\rm{diff}}) = A^{\sharp}( K)$ (so in fact $X$ is defined over $K$), so,  by  the Theorem of the Kernel, $A^{\sharp}(K^{\rm{diff}})  = A_{torsion}$, hence by Manin-Mumford, $X$ is a translate of an abelian subvariety of $A$.

\subsection{The characteristic $p$ case}\label{Char.p}

 Once we have quantifier elimination by Theorem \ref{QEthm}, the proof in
 the characteristic $p$ case is  substantially simpler than the
characteristic $0$ case, avoiding in particular  any recourse to the
socle theorem. We work with the notation introduced in Section
\ref{preliminaries}: $k = \FP$,  $K =
k(t)^{\rm{sep}}$, $A, X$ are over $K$, $\Gamma$ is contained in the prime-to-$p$ division points
of a finitely generated subgroup of $A(K)$ (so $\Gamma < A(K)$ too),
$X\cap \Gamma$ is Zariski-dense in $X$, and we take $\cal U$ to be a
saturated elementary extension of $K$. As before, $A^{\sharp}$ denotes
$A^{\sharp}(\cal U)$ and   $A$ is a sum of simple abelian subvarieties $A_{1},..,A_{n}$, all
defined over $K$. We know that each $A_{i}^\sharp$ is a $g$-minimal
subgroup of $A(K)$ (Fact \ref{FactAsharpbasic}(iii)). Moreover $A$ is assumed to have $k$-trace $0$.

\vspace{2mm}
\noindent
We now consider $\mathcal A := A^{\sharp} $ (respectively ${\mathcal A}_i
:= {A}_{i}^{\sharp}$) with their induced structure over $K$. 

\begin{Lemma} \label{Asharpinduced} The groups  $\mathcal A$ and
 ${\mathcal A}_i$ are connected groups of finite Morley rank, the
${\mathcal A}_i$ are $g$-minimal and
${\mathcal A}(K)$ is an elementary substructure of $\mathcal A$.
\end{Lemma}

\begin{proof}  We know by Fact \ref{basicAsharpcharp} that $A^\sharp$
  and the  ${A}_{i}^{\sharp}$ are connected groups, with relative Morley
  rank  and that the 
  ${A}_{i}^{\sharp}$ are $g$-minimal subgroups of $A^\sharp$.   By
  quantifier elimination of the induced structures (Theorem \ref{QEthm}),
it follows that $\mathcal A$ and the ${\mathcal A}_i$ are connected
groups of finite Morley rank, that each of the ${\mathcal A}_i$ 
remains a $g$-minimal subgroup of $\mathcal A$, and that $\mathcal A$ is
saturated. Consider the subset ${\mathcal A}(K)$ of $\mathcal A$. One
can check easily that as $K$ is an elementary substructure of $\cal
U$, quantifier elimination implies that ${\mathcal A}(K)$ is an algebraically closed subset
of $\mathcal A$. Furthermore, each ${\mathcal A}_i(K)$ is  infinite (it contains all the prime-to-$p$
torsion of $A_i$) and we can apply Corollary \ref{Corwagner} to conclude.
\end{proof}
 
\noindent We now complete the proof of Mordell-Lang:
            
 First here we again follow Hrushovski's proof to see that $\Gamma \cap
 X$ can be ``replaced'' by $E \cap X$ , for some $E$,   translate of
 $A^\sharp$, without losing the  Zariski denseness in $X$. 
Note that $p^{n+1}\Gamma$ has finite index in $p^{n}\Gamma$ for all
$n$. As $\Gamma$ meets $X$ in a Zariski-dense set, we find cosets
$D_{i}$ of $p^{i}\Gamma$ in $\Gamma$ such that $i<j$ implies
$D_{j}\subseteq D_{i}$, and each $D_{i}$ meets $X$ in a Zariski-dense
set. Hence we obtain a descending chain of cosets $E_{i}$ of
$p^{i}A(K)$ in $A(K)$, each meeting $X$ in a
Zariski-dense set. Passing to the saturated model $\cal U$, let $E$ be
the intersection of the $E_{i}$, $E$ is  a translate of $A^\sharp$
and by compactness $X\cap E$ is
Zariski-dense in $X$. 

At this point, Hrushovski appeals to the dichotomy theorem for thin
minimal types in separably closed fields. We replace this by
Manin-Mumford and the Theorem of the Kernel, but in order to do this,
we need to go down to the smaller field $K$ and we would need to know
that $E(K) \cap \Gamma$ is dense in $X$. But $K$ is no longer saturated,
and in fact, although $E$ is type definable over $K$,  $E(K)$ might be empty.
So we need to work with a slightly bigger translate than $E (K)$. 

We consider the two-sorted structure $M = ({\mathcal A}, {\mathcal E}) $, that is  $(A^\sharp ({\cal U}), E({\cal U}))$ with all the 
$K$-induced structure. It easily follows from Lemma \ref{Asharpinduced}  that 
$Th(M)$ has finite Morley rank and moreover that the sort $\mathcal A$
with induced structure has ${\mathcal A}(K)$ as an elementary
substructure. Let $M_{0} \prec M$ be prime (so atomic) over
${\mathcal A}(K)$. It follows that $M_{0}$ is of the form $({\mathcal A}(K),
{\mathcal E}_{0})$ for some elementary substructure ${\mathcal E}_{0}$ of
$\mathcal E$. Note that
${\mathcal E}_{0}$ is definably (without parameters) a principal homogeneous
space  for ${\mathcal A}(K)$. 
\newline
{\em Claim.}  $X\cap {\mathcal E}_{0}$ is Zariski-dense in $X$.
\newline
{\em Proof of claim.}  Suppose not. Then there is a proper subvariety
$Z$ of $X$ (defined over some field) such that $X\cap {\mathcal E}_{0} \subseteq
Z$. We may replace $Z$ by the Zariski closure of $X\cap {\mathcal
  E}_{0}$, and so we may assume that $Z$ is defined over ${\mathcal E}_{0}$. We now have that $X\cap
{\mathcal E}_{0} = Z\cap {\mathcal E}_{0}$.  Now $Z\cap \mathcal E $ viewed as a set definable in the
structure $\mathcal E$ (or $M$) is defined over ${\mathcal E}_{0}$, so as
${\mathcal E}_{0} \prec \mathcal E$ we
easily conclude that $X\cap \mathcal E  = Z\cap \mathcal E$, contradicting
Zariski-denseness of $X\cap \mathcal E$ in $X$. This completes the proof of the
claim.\qed 

\vspace{2mm}
\noindent
 Let $a\in X\cap {\mathcal E}_{0}$. Let $X_{1} = X-a$. Then $X_{1}\cap A^{\sharp}(K)$ is Zariski-dense in $X_{1}$. In particular $X_{1}$ is defined over $K$. Moreover using the Theorem of the Kernel, $X_{1}\cap A_{torsion}$ is Zariski-dense in $X_{1}$,  so by $MM$, $X_{1}$ is a translate of an abelian subvariety of $A$. This completes the proof.

\noindent Franck  Benoist\\
Laboratoire de Math\'ematiques d'Orsay\\
Univ. Paris-Sud, CNRS  \\
Universit\'e Paris-Saclay\\
91405 Orsay, France.\\
franck.benoist@math.u-psud.fr\\
\medskip 

\noindent Elisabeth Bouscaren\\
Laboratoire de Math\'ematiques d'Orsay\\
Univ. Paris-Sud, CNRS \\
Universit\'e Paris-Saclay\\
91405 Orsay, France.\\
elisabeth.bouscaren@math.u-psud.fr\\

\medskip 

\noindent Anand  Pillay\\
Department of Mathematics\\
University of Notre Dame\\
281 Hurley Hall\\
Notre Dame, IN 46556, USA.\\
 apillay@nd.edu\\


\vspace{1cm}










\begin{thebibliography}{99}

\bibitem{AV} D. Abramovich and F. Voloch, Towards a proof of the Mordell-Lang conjecture in positive characteristic, International Math. Research Notices 5 (1992), 103-115. 

\bibitem{BBP} F. Benoist, E. Bouscaren, and A. Pillay, Semiabelian varieties over separably closed fields, maximal divisible subgroups, and exact sequences, Journal of the Institute of Mathematics of Jussieu  DOI: 10.1017/S147474801400022X.

\bibitem{BenoistDelon} F. Benoist and F. Delon, Questions de corps de
  d\'efinition pour les vari\'et\'es ab\'eliennes en caract\'eristique
  positive, Journal de l'Institut de Math\'ematiques de Jussieu, vol. 7 (2008), 623--639.


\bibitem{bouscaren} E. Bouscaren, Proof of the Mordell-Lang Conjecture for Function fields, in  {\em Model theory
  and algebraic geometry}, edited by E. Bouscaren, Lecture Notes in
  Mathematics 1696, 2nd edition, Springer, 1999.

\bibitem{BouscarenDelon1} E. Bouscaren and F. Delon, Groups definable in separably closed fields, Transactions of the AMS 354 (2002), no. 3, 945-966. 

\bibitem{BD} E. Bouscaren and F. Delon, Minimal groups in separably closed fields,  J. Symbolic Logic 67 (2002), no. 1, 239–259.

\bibitem{BP} D. Bertrand and A. Pillay, A Lindemann-Weierstrass theorem for semiabelian varieties over function fields, Journal AMS, 23 (2010), 491-533. 

\bibitem{Buium-Annals} A. Buium,  Intersections in jet spaces and a conjecture of S. Lang, Annals of Math., 136 (1996), 557-567.

\bibitem{Buium-book} A. Buium, {\em Differential algebraic groups of finite dimension},  LNM 1506, Springer, 1992.

\bibitem{C} C. Corpet, Around the Mordell-Lang and Manin-Mumford conjectures in positive characteristic, preprint 2012.

\bibitem{Delon88} F. Delon, Id\'eaux et types sur les corps s\'eparablement clos, {\it Suppl\'ement au Bulletin de la SMF}, T116 (1988).

\bibitem{Delon} F. Delon, Separably closed fields, in {\em Model theory
  and algebraic geometry}, edited by E. Bouscaren, Lecture Notes in
  Mathematics 1696, 2nd edition, Springer, 1999.

\bibitem{hindryMM} M. Hindry, Autour d'une conjecture de Serge Lang, Inventiones Math. 94 (1988), 567-603. 

\bibitem{hindrymanchester}  M. Hindry, Introduction to abelian varieties and the Mordell-Lang conjecture, in {\em Model Theory and
    Algebraic Geometry}, E. Bouscaren (Ed.), Lecture Notes in
    Mathematics 1696, 2nd edition, Springer, 1999.



\bibitem{Hrushovski} E. Hrushovski, The Mordell-Lang conjecture for function fields, Journal AMS 9(1996), 667-690.

\bibitem{Hrushovski-MM} E. Hrushovski, The Manin-Mumford conjecture and the model theory of difference fields, Annals of Pure and Applied Logic, 112 (2001), 43-115.

\bibitem{HZ} E. Hrushovski and B. Zilber, Zariski geometries, Journal AMS 9(1996), 1-56.

\bibitem{lascar} D. Lascar, $\omega$-stable groups, in {\em Model theory
  and algebraic geometry}, edited by E. Bouscaren, Lecture Notes in
  Mathematics 1696, 2nd edition, Springer, 1999.

\bibitem{MMP} D. Marker, M. Messmer and A. Pillay, {\em Model theory of Fields}, Lecture Notes in Logic 5, second edition, ASL-AK Peters, 2003.

\bibitem{marker-pillay} D. Marker and A. Pillay, Differential Galois theory III: Some inverse problems, Illinois Journal of Mathematics, Vol. 41, Number 3 (1997), 453-461. 


\bibitem{OA} A. Omar Aziz, {\em Type definable stable groups and separably closed fields}, Ph. D. thesis, Leeds, 2012. 

\bibitem{Pillay} A. Pillay, Differential algebraic groups and the number of countable models, in {\em Model theory of Fields} above.

\bibitem{Pillay-book} A. Pillay, {\em Geometric Stability Theory}, Oxford University Press, 1996. 

\bibitem{Pillay-L} A. Pillay, The model-theoretic content of Lang's conjecture, {\em Model theory and algebraic geometry}, edited by E. Bouscaren, Lecture Notes in Mathematics 1696, Springer, 1996.

\bibitem{Pillay-ML} A. Pillay,  Mordell-Lang conjecture for function fields in characteristic $0$, revisited, Compositio Math. 140 (2004) 64-68. 

\bibitem{PZ} A. Pillay and M. Ziegler, Jet spaces of varieties over differential and difference fields, Selecta Math. 9 (2003), 579-599.

\bibitem{PR} R. Pink and D. R\"ossler, On $\psi$-invariant subvarieties of semiabelian varieties and the Manin-Mumford conjecture, J. Algebraic Geometry 13 (2004), 771-798. 

\bibitem{Poizatbook} B. Poizat, {\em Stable groups}, Mathematical Surveys and Monographs, Vol. 87, American Mathematical Society. 

\bibitem{Roessler-torsion} D. R\"ossler,  Infinitely $p$-divisible points on abelian varieties defined over function fields of characteristic $p>0$, Notre Dame Journal of Formal Logic 54 (2013),  no. 3-4, 579-589

\bibitem{R} D. R\"ossler, On the Manin-Mumford and Mordell-Lang conjectures in positive characteristic, Algebra and Number Theory 7 (2013), 2039-2057.

\bibitem{Sc} T. Scanlon,  A positive characteristic Manin-Mumford
  theorem, Compositio Math. 141 (2005),1351-1364.



\bibitem{W} F. Wagner, Fields of finite Morley Rank,
Journal of Symbolic Logic, 66(2001), 703-706

\bibitem{ziegler} M. Ziegler, Introduction to stability theory and Morley rank, in {\em Model theory and algebraic geometry}, edited by E. Bouscaren, Lecture Notes in Mathematics 1696, Springer, 1996.

\bibitem{Paul-Ziegler} P. Ziegler, Mordell-Lang in positive characteristic, Rend. Sem. Mat. Univ. Padova, vol 134 (2015), 93-131. 


\bibitem{Z}  B. Zilber, {\em Zariski geometries}, LMS Lecture series, CUP, 2011. 


\end{thebibliography}
\end{document}